\documentclass[12pt]{article}
\usepackage{fullpage}
\usepackage{amsmath,amsthm,amssymb,MnSymbol}
\usepackage{mathrsfs}
\usepackage{hyperref}
\usepackage{microtype}
\hypersetup{colorlinks=true, linkcolor=red,citecolor=blue,urlcolor=blue}

\theoremstyle{theorem}
\newtheorem{thm}{Theorem}
\newtheorem{lem}[thm]{Lemma}
\newtheorem{prop}[thm]{Proposition}
\newtheorem{cor}[thm]{Corollary}
\theoremstyle{definition}

\newtheorem{df}[thm]{Definition}
\newtheorem{rem}[thm]{Remark}
\newtheorem{ex}[thm]{Example}

\newcommand{\N}{\mathbb{N}}
\newcommand{\X}{\mathcal{X}}
\newcommand{\D}{\mathcal{D}}
\newcommand{\F}{\mathcal{F}}

\newcommand{\PV}{\mathsf{PV}}
\newcommand{\ZFC}{\mathsf{ZFC}}
\newcommand{\PA}{\mathsf{PA}}
\renewcommand{\P}{\mathsf{P}}
\newcommand{\NP}{\mathsf{NP}}
\newcommand{\coNP}{\mathsf{coNP}}

\newcommand{\im}{\mathrm{im}}
\newcommand{\uf}{ }
\newcommand{\bit}{\textit{bit}}

\renewcommand{\le}{\leqslant}
\renewcommand{\ge}{\geqslant}
\newcommand{\res}{{\upharpoonleft}}
\renewcommand{\ldots}{...}

\begin{document}

\title{\bf Nonstandard combinatorial ultrafilters\\ and a new local witness model}

\author{Johannes Heil\\\small Universit\"at Passau,
  Germany\\[-1ex]\small\url{heil@fim.uni-passau.de} 
 \and Moritz
  M\"uller\\\small Universit\"at Passau,
  Germany\\[-1ex]\small\url{moritz.mueller@uni-passau.de} }

\maketitle

\begin{abstract} A central object in the theory of proof complexity generators is Kraj\'i\v{c}ek's local witness model. We construct a new such model as a definable ultrapower using a definable version of a selective ultrafilter over a  model of arithmetic. We transfer some set theory of combinatorial ultrafilters to the definable nonstandard setting. \end{abstract}

\section{Introduction}
This work concerns the model theory of bounded arithmetics, in particular, the universal variant $\forall\PV$ (going back to \cite{lipton}) of Cook's theory $\PV$ \cite{cook}. Its language, also denoted $\PV$, 
 contains $<$ and symbols for all polynomial time functions. $\forall\PV$ is the universal theory of the standard $\PV$-structure $\N$ (that interprets the symbols straightforwardly).

The contribution of this work consists in the construction of a new so-called local witness model \cite[Chapter 31]{kraforce}. This introduction explains the importance of such models for propositional proof complexity. We refer to the monographs  \cite{krabuch,cn,krared} 
for the connection of bounded arithmetics and propositional proof systems.

\paragraph{Proof complexity generators} 

It is somewhat embarrassing for mathematical logic that it is unknown whether a usual textbook calculus for propositional logic based on finitely many axioms and inference rules has {\em short} proofs of all tautologies, i.e., of size polynomial in the size of the tautology. Establishing 
proof size lower bounds for such {\em Frege} systems and stronger ones is the central problem of proof complexity; \cite{cookreck} is a foundational paper of the subject. 
The question is tightly related to the fact (cf.~\cite{bottom}) that  mathematical logic lacks general methods to prove independence of $\Pi_1$ sentences, even from rather weak theories like Buss' bounded arithmetics \cite{bussthesis}.

The most promising approach to find hard tautologies is via proof complexity generators and goes back to \cite{alekh,krawphp}. Their theory is laid out in the recent monograph~\cite{krapcg}. 

A {\em generator} is a function $g:\N\to\N$ with $|g(x)|>|x|$ for all $x\in\N$; for simplicity we restrict attention to $|g(x)|=|x|{+}1$. Let $b$ have length $n{+}1$. If $g$ is polynomial time or, more generally, has bitgraph in $\NP$, it is straightforward to write a propositional formula $\tau(g)_b$ that expresses of its $n$ variables (plus some auxiliary ones) that they are the bits of a length~$n$ number~$x$ with $g(x)\neq b$. Thus, for $b$  outside the range $\im(g)$ of $g$, we get tautologies $\tau(g)_b$. 

An important example is  a Nisan-Wigderson generator \cite{nw} based on a hard-core bit \cite{gl} of a sufficiently  secure one-way permutation.
Razborov \cite{raz} conjectured  that for this (and related) $g$ 
the tautologies $\tau(g)_b$ are hard for Extended Frege. 
Moving to arbitrary proof systems, Kraj\'i\v{c}ek~\cite{krajml} puts forward the conjecture:
\begin{quote}\em
There exist a generator $g$ with bitgraph in $\NP$ such that $\im(g)$ intersects every infinite $R\in\NP$.
\end{quote}
This implies lower bounds for every proof system, and hence $\NP\neq\coNP$ \cite{cookreck}: note the set $R$ of those  $b$ such that 
$\tau(g)_b$ has a short proof in the system, is in $\NP$ and disjoint from $\im(g)$ by soundness -- so must be finite. 

These conjectures have philosophical importance in that they concern the very nature of proof complexity: Kraj\'i\v{c}ek closes his monograph asking ``Can the proof complexity of some formulas be traced to the computational complexity of associated computational tasks?''~\cite[p.477]{krared}. 
This is also the lead question of Razborov's~\cite{raz}.
In the same vein, Pudl\'ak's conjectures on feasible incompleteness \cite{pubull} ``say that a possible cause of incompleteness is computational complexity.'' \cite[p.563]{pudlakbook}

\paragraph{Local witness models} For the example $g$ above,
Kraj\'i\v cek showed that his conjecture is consistent with $\forall\PV$. More precisely, he constructed the following for 
every infinite $R\in\NP$:

\begin{df} A {\em local witness model for $g$ and $R$} is a model of $\forall\PV+\exists x \text{``}g(x)\in R\text{''}$.\end{df}

The definition of the formula $\text{``}g(x)\in R\text{''}$ is straightforward and spelled out in Section~\ref{sec:lwm}. This consistency of the conjecture can be interpreted as evidence for its truth
(cf.~\cite[p.476]{krared} or \cite[p.326]{krabuch}).
More importantly (see Remark~\ref{rem:coll}), $R\cap\im(g)\neq\emptyset$ follows from the existence of such a model that satisfies  sharply bounded collection  for quantifier-free $\PV$-formulas (likely unprovable in $\forall\PV$ \cite{cookthapen,ondra}). We refer to \cite[Section~5]{krajml} for a discussion, and to \cite{kraneu} for a route side-stepping collection. 

We distinguish two steps in Kraj\'i\v cek's construction. The 1st step is complexity theoretic and derives from a sufficiently secure one-way permutation a less standard hardness hypothesis~(H) -- explained next.
 The proof is ingenuous with further applications in \cite{pich,ps,igor,ondra}. 
 
For  $b\notin\im(g)$ of length $n{+}1$ consider the computational task ($\textup{T}_b$): on input $x$ of length~$n$ find $i\le n$ such that $\bit(g(x),i)\neq \bit(b,i)$, i.e., the $i$th bits of  $g(x)$ and  $b$ differ.
The computational  model are student-teacher computations. These proceed in rounds: in each round, the student computes a candidate solution $i\le n$, and the teacher responds with a counterexample, namely a witness $v$
 for $\bit(g(x),i)= \bit(b,i)$ (an 
$\NP$-property of~$(x,i)$).
If there is no counterexample, the student succeeds.
 If this happens for all inputs~$x$ of length~$n$ 
 and all possible answers of the teacher 
 within $\le c$ rounds, we say the student {\em solves} ($\textup{T}_b$) 
 {\em in $c$ rounds}.
Formally, such a student is a sequence $C_0,\ldots, C_{c-1}$ of Boolean circuits; $C_k$ computes the bits of the candidate solution $i\le |n|$ of round $k$ from the bits of $x$ and the counterexamples $v_0,\ldots, v_{k-1}$ of preceeding rounds; the size of the student  is the sum of the sizes of the $C_k$.
\begin{quote} 
(H) \em For all $c,d\in\N$, all sufficiently large $n$ and all $b\notin \im(g)$ of length $n{+}1$ the following holds. Every student that solves \textup{($\textup{T}_b$)} in $c$ rounds has  size  $\ge n^d$.
\end{quote}

The 2nd step in \cite{krajml} uses a familiar model-theoretic argument based on KPT-witnessing~\cite{kpt} to derive from (H) the existence of a local witness model.
With some extra work in the 1st step,  Kraj\'i\v{c}ek \cite{kraforcenew} constructs, quite explicitly, some Boolean valued local witness model by forcing with random variables~\cite{kraforce} and in \cite[p.118]{krapcg} expresses the hope it might satisfy collection. 
%
Both constructions give local witness models that 
are cofinal extensions of some small canonical model of $\forall\PV$, namely a suitable cut of a nonstandard model~$M$ of true arithmetic. Also a third construction \cite[Section 8.6]{krapcg} for a different type of generator $g$ and a different hardness assumption produces cofinal extensions.

We invest some extra work in the 2nd  step to prove:

\begin{thm}\label{thm:lwm} Assume $R\in\NP$ is infinite and $g$ is a generator with bitgraph in $\NP$ that satisfies the hardness assumption {\em (H)}. Then there exists a  local witness model for $g$ and $R$ that is a proper 1-elementary end extension of $M$.
\end{thm}

The proof 
has two interesting corollaries. The first (Corollary~\ref{cor:lwm}) states  the existence of a
 local witness model for $g$ and all infinite $R\in\NP$ simultaneously -- under a slight strengthening of  (H). The second (Corollary~\ref{cor:lwm2}) concerns the question whether
 (H) can be weakened to a uniform variant (uH) concerning uniform polynomial time students.
This does not seem to be the case for  Kraj\'i\v{c}ek's construction (cf.~\cite[p.113]{krapcg}).
Our proof gives the variant of Theorem~\ref{thm:lwm} with (uH) replacing (H) and for 
any $R\in\NP$ that is {\em feasibly infinite}, i.e., $\forall\PV$-provably infinite. This shows 
the consistency under (uH) of the weakening of Kraj\'i\v{c}ek's conjecture to feasibly infinite $R\in\NP$ --  recently  studied in \cite{krabull} (see also~\cite[Sections 8.7, 10.3]{krapcg}).


 The proof of Theorem~\ref{thm:lwm} proceeds similarly to forcing with random variables~\cite{kraforce} but uses the somewhat more 
flexible set-up of definable ultrapowers. The method of ultrapowers has been used in the model theory of $\forall\PV$ in~\cite{pupower,krapower,garlik,byd} and, in fact, is universal in some sense  (see \cite[Theorem~1.4]{byd}).
The core of our  proof is the construction of a certain combinatorial ultrafilter. We elaborate the underlying ultrafilter theory to some  extent beyond  what is required for this proof. 
We now describe our motivation and contribution.

\paragraph{Nonstandard combinatorial ultrafilters}
Forcing with random variables \cite{kraforce} produces Boolean valued models. 
It is a classic \cite[Section~5]{rs}   that such models can be collapsed to usual two-valued
models by factoring the Boolean algebra by a suitable ultrafilter. The current work originates from the hope that the choice of the ultrafilter is a further angle to construct interesting models. This is not investigated in \cite{kraforce}. 

An idea is to use combinatorial ultrafilters studied in set theory~\cite[Chapter 11]{halb}; their impact
on ultrapowers of $\N$ (in the full language) is long known \cite{puritz,blass1}.
 For example, {\em Ramsey} ultrafilters (over $\N$) contain homogeneous sets for all 2-colorings of unordered pairs, and {\em selective} ultrafilters contain for every function a set where the function is constant or injective. A basic, non-trivial result is the equivalence of these two notions (and others). 

A minor obstacle is that such ultrafilters
 may not exist  -- but it is consistent (with~$\ZFC$) to assume they do (see e.g. \cite[Proposition 15.15]{halb}). 
A more serious obstacle is that, instead of~$\N$ we have a nonstandard $M$ and it is not straightforward  to transfer the  set theoretic concepts and results. We refer to \cite{thesis} for a 
thorough discussion.  To give an impression,  the mentioned filters must contain a set of order type the cofinality $\kappa$ of $M$. The mentioned equivalence can be shown for such filters that are  $\kappa$-closed. Then~$\kappa$ is measurable, if uncountable, and this is an assumption of higher consistency strength that we do not know how to avoid. E.g., Kraj\'i\v{c}ek's  $M$ from \cite{kraforce} is $\aleph_1$-saturated, so has uncountable cofinality~$\kappa$.

To avoid esoterical assumptions on $M$, we move to the definable setting:  ultrafilters live in the Boolean algebra   $\D(M)$ of definable sets and we define e.g.~d-Ramsey and d-selectivity with respect to definable 2-colorings, resp., definable~functions. This seems ``right'': Theorem~\ref{thm:ramsey} faithfully transports the proof of the mentioned equivalence -- with some natural extra assumption (namely definability as types). 
Furthermore, already $\ZFC$ proves that these ultrafilters exist (see Theorem~\ref{thm:exist}).

Related (see \cite{blass2}) is  Gaifman's  seminal work \cite{gaif} on types of Peano arithmetic $\PA$; we refer to \cite{kossak} for a survey. 
E.g., \cite[Theorem 3.7]{gaif} shows the existence of types of $\PA$ all of whose extensions to parameters from any model $M$ are d-selective (we blur the distinction between types and ultrafilters in $\D(M)$). Realizing 
such types produces minimal end extensions  \cite[Theorem~3.6]{gaif}),  an argument we use to strengthen Theorem~\ref{thm:lwm} (see Theorem~\ref{thm:stronglwm}). 

\section{Preliminaries}

\subsection{The ground model $M$}

 Let $L$ be a  language containing $\PV$ and $\N$ be an $L$-structure (with universe  $\N$) that interprets  the symbols from $\PV$ canonically.  
 
 Apart 
 $<,+,\cdot$ we shall use the following: the smash function is $n\# m:=2^{|n|\cdot |m|}$ and $|n|$ is the length of the binary representation of $n$, i.e., $n=\sum_{i<|n|}\bit(n,i)\cdot 2^i$  
 where $\bit(n,i)\in\{0,1\}$ for $i<|n|$ is the $i$th bit of $n$
 and we agree $\textit{bit}(n,i)=0$ for $i\ge |n|$. We have $\langle n,m\rangle $ for Cantor's pairing and write $\langle \ell, n,m\rangle:=\langle\langle \ell, n\rangle,m\rangle$ etc. 
Every $n\in\N$ codes 
a length $\textit{lh}(n)$ sequence $((n)_0,\ldots, (n)_{\textit{lh}(n){-}1})$ with
$(n)_i$ denoting components and we agree $(n)_i=0$ for $i\ge\textit{lh}(n)$.

On the formal level where $x=|y|$ we  write $2^x$ for $y\# 1$ and $2^{x^2}$ for $y\# y$ etc.

 \medskip

Throughout, we fix  a (not necessarily proper) elementary extension $M$ of $\N$. 
We view $M$ as an $L(M):=L\cup M$ structure that interprets every 
{\em parameter}
$a\in M$, viewed as a constant, by itself. We write  $s^M$ for the interpretation of $s\in L$ in $M$ but omit the superscript for certain $s\in\PV$ like the above. E.g., for $a,b\in M$ writing $a<b$ means $(a,b)\in\ <^M$.

Let $ \mathcal D(M)$ denote the Boolean algebra of subsets of $M$ that are definable in $M$ (with parameters from~$M$, i.e., as an $L(M)$-structure).  It consists of the sets 
$$\varphi(M):=\{a\in M\mid M\models\varphi(a)\}
$$ 
for $L(M)$-formulas $\varphi(x)$; note, by tuple-coding, $\varphi(M)=\psi(M,a)$ for some $L$-formula $\psi(x,y)$
and a single parameter $a\in M$.

A family $(X_a)_{a\in M}$ 
is {\em uniformly definable} if $\{\langle a,b\rangle\mid b\in X_a\}\in\D(M)$; it is {\em non-increasing} if $X_a\supseteq X_b$ for all $a\le b$. For families $(X_a)_{a\le b}$
with $b\in M$ these modes of speech are similarly explained. 
A set $X\in\D(M)$ is {\em bounded} if 
$$X\subseteq [b]:=\{a\in M\mid   a< b\}
$$ for some $b\in M$. 
A family $\X\subseteq\D(M)$ is {\em unbounded} if all $X\in\X$ are unbounded.

 The following is easy to see
(see~\cite[Lemma~2.17]{byd}):

\begin{lem}\label{lem:uf} A family $\X\subseteq\D(M)$ is contained in an unbounded ultrafilter in $\D(M)$ if and only if $\X$ has the {\em unbounded intersection property}: $\bigcap\X_0$ is unbounded for all finite $\X_0\subseteq\X$.
\end{lem}

\subsection{Definable ultrapowers}

Let $\F(M)$ denote
 the set of definable functions $f:M\to M$, i.e.,
$\{\langle a,f(a)\rangle\mid a \in M\}\in\D(M)$.
The restriction of $f$ to $X\in\D(M)$ is denoted  $f\res X$; it is 
{\em bounded-to-one}
if $f^{-1}(b)\cap X =\{a\in X\mid f(a)=b\}$ is bounded for all~$b\in M.$

Let $F\subseteq\F(M)$ be  {\em $L(M)$-closed}:  for every function symbol $f\in L(M)$, say $r$-ary, and all $\bar\alpha=(\alpha_1,\ldots, \alpha_r)\in F^r$ also $f\circ\bar\alpha\in F$; this function  maps
$a\in M$ to $ f^M(\bar\alpha(a))$ where $\bar\alpha(a):=((\alpha_1(a),\ldots,\alpha_r(a))$. For a constant $b\in M$ this means that
$F$ contains  the constant function $\alpha_b$ that maps all $a\in M$ to $b$. 


For  an ultrafilter $\uf U$ in $\D(M)$ we define the $L(M)$-structure $F/\uf U$ as follows. Its    universe  is
$\{\alpha^{\uf U}\mid\alpha\in F\}$ where $\alpha^{\uf U}\subseteq F$ is the set of $\beta\in F$ such that
$\{a\in M\mid \alpha(a)=\beta(a)\}\in \uf U.$
A function symbol $f\in L(M)$ is interpreted in $F/\uf U$ by $\bar\alpha^{\uf U}\mapsto(f\circ\bar\alpha)^{\uf U}$ where
$\bar \alpha^{\uf U}:=(\alpha_1^{\uf U},\ldots,\alpha^{\uf U})$. In particular,
$\alpha^{\uf U}_b\in F/\uf U$ for all $b\in M$.
The relation symbol 
$<\in\PV$ is interpreted in $F/\uf U$ setting for all $\alpha^{\uf U},\beta^{\uf U}\in F/\uf U$,
$$
\alpha^{\uf U}<^{F/\uf U}\beta^{\uf U}\ \Longleftrightarrow\ \{a\in M\mid \alpha(a)< \beta(a)\}\in\uf U.
$$
Similarly for other relation symbols in $ L$. It is easy to check that $F/\uf U$ is well-defined. 

The following is straightforward (see e.g.~\cite[Proposition~2.10]{byd}):

\begin{lem} \label{lem:PV}  For every quantifier-free $L(M)$-formula $\varphi(\bar x)$ and every tuple $\bar\alpha$ from $ F$ we have
$$
F/\uf U\models\varphi(\bar\alpha^{\uf U})\ \Longleftrightarrow\ \{a\in M\mid M\models\varphi(\bar\alpha(a))\}\in\uf U.
$$
In particular, 
 $M\cong M':=\{\alpha_a^{\uf U}\mid a\in M\}\preceq_1F/\uf U$. 
\end{lem}

Here, $M\preceq_1 N$ means $N$ is a 1-elementary extension of $M$, i.e., $M\subseteq N$ and $N$
satisfies every universal $L(M)$-sentence that is true in $M$; $M\prec_1 N$ additionally indicates $N\neq M$.


\section{Nonstandard combinatorial ultrafilters}\label{sec:uf}



\subsection{Definable ultrafilters}\label{sec:pfip}

The  nonstandard definable setting requires the  concept of a definable ultrafilter.  It is remarkable that this concept, central in tame model theory, in particular of stable theories (see e.g.~\cite[Section~11.4]{poizat}, \cite[Section~8.3]{tz}),  seems to have originated in the model theory of foundational theories, namely Gaifman's work \cite{gaif} on $\PA$ (cf.~\cite[p.237]{poizat}).

\begin{df}\label{df:dfble} An ultrafilter $\uf U$ in $\D(M)$ is {\em definable} if for every $L(M)$-formula 
$\varphi(x,z)$,
$$
\big\{a\in M\mid \varphi(M, a)\in\uf U\big\}\in\D(M);
$$
\end{df}

Following \cite[p.130]{tz} we write $\textup{d}_Ux\varphi(x,z)$ for an $L(M)$-formula $\psi(z)$ defining this set.

\begin{lem}\label{lem:dfblepfip} Definable ultrafilters in $\D(M)$ have the   {\em pseudo-finite intersection property (pfip)}: for every $b\in M$ and  every uniformly definable  $(X_a)_{a\le b}$ in $\uf U$ we have $\bigcap_{a\le b}X_a\in\uf U$.
\end{lem}

\begin{proof} 
Let $\varphi(x,y)$ be an $L(M)$-formula such that $\varphi(x,c)$ defines $Y_c:=\bigcap_{a\le c}X_a$ for all $c\le b$. We have to show $Y_b\in U$, i.e., $M\models\textup{d}_Ux\varphi(x,b)$. This follows by induction in $M$: $Y_0=X_0\in U$, so $M\models \textup{d}_Ux\varphi(x,0)$; if $c<b$ and $M\models\textup{d}_Ux\varphi(x,c)$, then $Y_c\in U$, so $Y_c\cap X_{c+1}=Y_{c+1}\in U$, i.e., $M\models\textup{d}_Ux\varphi(x,c{+}1)$.
\end{proof}

\begin{lem}\label{lem:psfip} An ultrafilter $\uf U$ in $\D(M)$ has the pfip if and only if
for every bounded $f\in \F(M)$ there is an $a\in M$ such that $f^{-1}(a)\in\uf U$.
\end{lem}

\begin{proof}$\Rightarrow:$ assume $f$ is definable with values $\leq b$ and $f^{-1}(a)\notin\uf U$ for all $a\in M$. Then $(M\setminus f^{-1}(a))_{a \leq b}$ is uniformly definable and in $\uf U$ with intersection $\emptyset\notin\uf U$.

$\Leftarrow:$ assume $(X_a)_{a \leq b}$ is uniformly definable and  in $\uf U$ with $X:=\bigcap_{a \leq b}X_a\notin\uf U$. Let $f:M\to M$ map $x\in X$ to $b {+} 1$ and every $x\notin X$ to 
 $\min \{a \leq b\mid x\notin X_a \}$. 
Then $f$ is definable with values $\le b {+} 1$. But $f^{-1}(b + 1)=X\notin\uf U$ and $f^{-1}(a)\notin\uf U$ for every $a \leq b$ because
$f^{-1}(a)\subseteq M\setminus X_a\notin\uf U$. 
\end{proof}

\subsection{d-P-points}\label{sec:Ppoint}

\begin{df}
 For $X,Y\subseteq M$ let $X\subseteq^* Y$ mean that $X\setminus Y$ is  bounded.
A {\em  pseudo-intersection of $\X\subseteq\D(M)$} is an unbounded set
$Y\in\D(M)$ with $Y\subseteq^* X$ for all $X\in\X$.

 An ultrafilter $\uf U$ in $\D(M)$ is a {\em d-P-point} if every uniformly definable family $(X_a)_{a\in M}$ in~$\uf U$ there is a pseudo-intersection in $\uf U$.
\end{df}

\begin{lem}\label{lem:psint} Every  uniformly definable non-increasing unbounded family $(X_a)_{a\in M}$  has a pseudo-intersection in $\D(M)$.\end{lem}

\begin{proof} 
Let $f:M\to M$ map $a\in M$ to $\min (X_a\setminus[a])$. Then $f$ is definable and so is its image~$X$. Clearly, $X$ is unbounded. We claim $X\subseteq^* X_a$ for all $a\in M$. 

Let $a\in M$. If $a\le b$ then $f(b)\in X_b\subseteq X_a$. Thus, if $f(b)\notin X_a$, then $b<  a$. Hence 
$X\setminus X_a\subseteq[c]$ where $c:=1{+}\max_{b<a}f(b)$; note $c$ exists in $M\equiv \N$.
\end{proof}



The following has been shown for $M=\N$ in \cite[Lemma 3]{blass2}. We include the short argument for completeness.

\begin{lem}\label{lem:Ppoint} An ultrafilter $\uf U$ in $\D(M)$ is a d-P-point if and only if  $\uf U$ is unbounded and
for every $f\in \F(M)$ there is $X\in\uf U$ such that $f\res X$ is constant or bounded-to-one.
\end{lem}

\begin{proof} $\Rightarrow$: let $f\in\F(M)$ and assume $f^{-1}(b)\notin\uf U$ for all $b\in M$.
Then $(M\setminus f^{-1}(b))_{b\in M}$ is uniformly definable and in $\uf U$. Let $X\in\uf U$ be a pseudo-intersection. But $X\subseteq^*M\setminus f^{-1}(b)$ means that $X\cap f^{-1}(b)$ is bounded.  

$\Leftarrow$: 
let $(X_a)_{a\in M}$ be uniformly definable and in $U$. Then so is $(X_a\setminus[a])_{a\in M}$; write $Y_a:=X_a\setminus[a]$. It suffices to find a pseudo-intersection of $(Y_a)_{a\in M}$ in $U$. 

Let $f\in\F(M)$ map $x\in M$ to the minimal $a\in M$ with $x\notin Y_a$. Choose $Z\in U$ such that $f\res Z$ is bounded-to-one or constant.
If $f\res Z$ was constant, say value with $a$, then $Z\cap Y_a=\emptyset$ contradicting $Z,Y_a\in U$. Hence $f\res Z$ 
is bounded-to-one. Then $Z\subseteq^* Y_a$ for all $a\in M$.
Indeed, $f\res(Z\setminus Y_a)$ has values in $[a{+}1]$ so is bounded-to-one  with bounded image; by $M\equiv\N$, its domain $Z\setminus Y_a$ is bounded.\end{proof}


 
 

\subsection{d-Ramsey ultrafilters}\label{sec:ramsey}

The following is an $M$-analogue of the set theoretical characterization of Ramsey ultrafilters over $\N$, as laid out e.g.~in~\cite[Chapter 11]{halb}.

We write
$\forall x,y\in X\ldots$  for $\forall x\forall y(\varphi(x)\wedge\varphi(y)\to\ldots)$ where $\varphi(x)$ defines $X$.

\begin{thm}\label{thm:ramsey} Let $\uf U$ be an unbounded  ultrafilter in $\D(M)$. Then
$
1\Rightarrow 2\Leftrightarrow 3\Rightarrow 4.
$
Moreover, $4\Rightarrow 1$ holds if $U$ is definable.
\begin{enumerate}\itemsep=0pt
\item $\uf U$ is {\em d-Ramsey}:  for all $L(M)$-formulas $\varphi(x,y)$ there is $X\in\uf U$ such that
$$
M\models \forall x,y \in X(x<y\to \varphi(x,y))\vee \forall x,y\in X( x<y\to \neg\varphi(x,y)).
$$ 

\item $\uf U$ is {\em d-selective}: for all $f\in\F(M)$ there is $X\in \uf U$ such that $f\res X$ is constant or injective.

\item $\uf U$ is a d-P-point and a {\em d-Q-point}: for all  bounded-to-one $f\in \F(M)$ there is $X\in \uf U$ such that $f\res X$ is injective.

\item $\uf U$ is {\em d-diagonalizing}: for every uniformly definable, non-increasing $(X_a)_{a\in M}$ in 
$\uf U$ there is $X\in \uf U$ such that for all $a\in X$ we have $X\setminus[a{+}1]\subseteq X_a$.
\end{enumerate}
\end{thm}

\begin{proof} 
$1\Rightarrow 2:$ given $f\in\F(M)$ let $\varphi(x,y)$ be $f(x)=f(y)$ and choose $X$ according to (1). If the 1st disjunct holds, then $f\res X$ is constant. If the 2nd disjunct holds, then $f\res X$ is injective. 

\medskip

$2\Rightarrow 3$: $\uf U$ is clearly a d-P-point by Lemma~\ref{lem:Ppoint}, and it is a d-Q-point because 
a bounded-to-one $f$ cannot be constant on a set in $\uf U$, being unbounded. 

\medskip

$3\Rightarrow 2$: let $f\in\F(M)$ and choose $X\in\uf U$ such that $f\res X$ is constant or bounded-to-one (Lemma~\ref{lem:Ppoint}). In the 1st case we are done, so assume the 2nd. Let $f'$ agree with $f$ on $X$ and be the identity otherwise.
Then $f'\in \F(M)$ is bounded-to-one and the d-Q-point property gives $Y\in \uf U$ such that
$f'\res Y$ is injective. Then so is $f'\res (X\cap Y)=f\res (X\cap Y)$ and $X\cap Y\in\uf U$.

\medskip

$2\Rightarrow 4$:
 let $(X_a)_{a\in M}$ be uniformly definable, non-increasing and in $\uf U$. If $Z:=\bigcap_{a\in M}X_a\in\uf U$, we are done. Assume $Z\notin\uf U$. 
Let $f\in\F(M)$ map every $x\in Z$ to 0 and every  $x\notin Z$ to
$\min \{a\in M\mid x\notin X_a \}$; note $f^{-1}(0)=Z\cup(M\setminus X_0)\notin\uf U$ and for $a>0$,
$$\textstyle
f^{-1}(a)= \bigcap_{b<a}X_b\setminus X_{a}\notin \uf U.
$$
 
Choose $X\in\uf U$ according to (2). Then $X$ is unbounded and $f\res X$ is injective. Given $a\in M$ we have $f(x)>a$ for large enough $x\in X$, say for $x\ge g(a)$ where $g\in \F(M)$. We can assume~$g$ is increasing and $g(0) > 0$ (otherwise replace by $a\mapsto 1{+}\sum_{b\le a}g(b)$). But $f(x)> a$ implies $x\in  X_a$. 
This is the crucial property of $g$: for all $a\in M$,
$$
X\setminus[g(a)]\subseteq X_{a}.
$$ 

Define $h(y)$ by iterating $g$, i.e., $h(0):=0$ and $h(y{+}1):=g(h(y))$. The crucial property  of~$h$ is that for all $a,b\in M$ (using that $g$ is increasing),
$$
a<h(b{+}1)\ \Longrightarrow\ g(a)< g(h(b{+}1))=h(b{+}2).
$$

Define $h^*(x)$ as the $a\in M$ such that $x\in[h(a),h(a{+1}))$. By (2) choose $X^*\in \uf U$ such that $h^*\res X^*$
is injective. Clearly, $h^*\res X^*$ is increasing.

Partition $X^*=X^*_0\cup X^*_1$ where $X_0^*$ and $X^*_1$ collect those $x\in X^*$ with even, resp., odd~$h^*(x)$. Assume $X^*_0\in \uf U$ (the case $X^*_1\in\uf U$ is analogous).

We claim $X_0^*\cap X\in \uf U$ witnesses (4), that is, for all $a\in X_0^*\cap X$,
$$
(X_0^*\cap X)\setminus[a{+}1]\subseteq X_a.
$$
Let $a<x\in X_0^*\cap X $. We have to show $x\in X_a$. 
Since $h^*\res X^*$ is increasing,  $a<x$ implies $b:=h^*(a)<h^*(x)=:c$. As $b,c$ are both  even, 
$b{+}2\le c$. As $h$ is increasing, $h(b{+}2)\le h(c)$. By 
definition of $h^*$ we have $h(c)\le x$ and $a<h(b{+}1)$. By the crucial property of $h$ we get
$g(a)<h(b{+}2)$. Thus, $g(a)< x$ and $x\in X\setminus[g(a)]$. Then $x\in X_a$ by the crucial property of $g$.

\medskip

Moreover:  assuming (4) and that $\uf U$ is unbounded and definable, we verify~(1).
Let $\varphi(x,y)$ be an $L(M)$-formula. For $a\in M$ let $X_a:=\varphi(a,M)$ if this set is in $\uf U$, and $X_a:=\neg\varphi(a,M)$ otherwise. By the definability of $\uf U$, $(X_a)_{a\in M}$ is uniformly definable, and it is clearly in $\uf U$. By the pfip (Lemma~\ref{lem:dfblepfip}), so is $(Y_a)_{a\in M}$ for $Y_a:=\bigcap_{b\le a}X_b$. Further,
$(Y_a)_{a\in M}$ is non-increasing.

Choose $Y\in\uf U$ according to (4), i.e., for all $a\in Y$ 
$$
Y\setminus[a{+}1]\subseteq Y_a\subseteq X_a.$$ 
Then  either $Y\setminus[a{+}1]\subseteq\varphi(a,M)$ or $Y\setminus[a{+}1]\subseteq\neg \varphi(a,M)$. 
Choose $Z\subseteq Y$ with $Z\in \uf U$ such that for all $a\in Z$ the same case happens, say 
$$Y\setminus[a{+}1]\subseteq\varphi(a,M).$$ 
Then $Z$ witnesses (1): let $a,b\in Z$ with $a<b$; then $b\in Y\setminus[a{+}1]$, so $M\models\varphi(a,b)$. 
\end{proof}

We do not know whether the d-Ramsey property implies its version for coloring 3-element sets.
 This can be verified for 
another natural, albeit slightly more technical way to transfer the set theoretic concepts from $\N$ to $M$ that goes back to Blass~\cite{blass2} (he defined (2) below in the case $M=\N$).

\begin{thm} Let $\bar x=(x_0,\ldots, x_{k-1})$ for $k\ge 2$. Let $U$ be an unbounded ultrafilter in $\D(M)$. The following are equivalent.
\begin{enumerate}\itemsep=0pt
\item For every $L(M)$-formula $\varphi(\bar x,y)$ there is an $L(M)$-formula $\psi(x,y)$ such that for all $b\in M$, $X_b:=\psi(M,b)\in U$ is {\em homogeneous} for $\varphi(\bar x,b)$: either all increasing $\bar c\in X_b^k$ satisfy $\varphi(\bar x,b)$ in $M$ or all increasing $\bar c\in X_b^k$ satisfy $\neg\varphi(\bar x,b)$ in $M$. 
\item $U$ is {\em uniformly d-selective}: for all $L(M)$-formulas $\varphi(x,y,z)$ there is 
an $L(M)$-formula $\psi(x,z)$ such that for all $b\in M$ with $\varphi(x,y,b)$ defining a function $f_b:M\to M$ we have $X_b:=\psi(M,b)\in U$ and $f_b\res X_b$ is constant or injective. 
\end{enumerate}
\end{thm}

\begin{proof} 
$1\Rightarrow 2$ follows as in the previous proof; note (1) for $k\ge 2$ implies (1) for $k=2$. 

Conversely, we assume (2) and prove (1) for all $k>0$ including $k=1$. We use induction on~$k$. The case for $k=1$ follows
from the
 definability of $U$, proved as Lemma~\ref{lem:unidef} below: given $\varphi(x,y)$, set
$$
 \psi(x,y)\ :=\ (\varphi(x,y)\leftrightarrow\textup{d}_Ux\varphi(x,y)).
$$
 
We assume (1)-for-$k$
and show (1)-for-$(k{+}1)$. Given $\varphi(x_0,\ldots,x_k,y)$ and $a,b\in M$, (1)-for-$k$ gives a formula
$\psi(x,y,z)$ such that $X^b_a:=\psi(M,a,b)\in U$ is homogeneous for $\varphi(a,x_1,\ldots, x_k,b)$. 

We claim that  there exists an $L(M)$-formula  $\delta(x,y)$ such that for all $b\in M$ we have $Y_b:=\delta(M,b)\in U$ and for all $a\in Y_b$,
$$
Y_b\setminus [a{+}1]\subseteq X^b_a.
$$ 

The claim implies the theorem: given $a\in Y_b$, either $M\models\varphi(a,\bar c,b)$ for all  increasing $\bar c\in Y^k_b\setminus[a{+}1]$, or  $M\models\neg\varphi(a,\bar c,b)$  for all  increasing $\bar c\in Y^k_b\setminus[a{+}1]$. Hence, $Y_b$ splits into $Y^0_b\cup Y^1_b$ according to these two cases. Both  $Y^0_b$ and $Y^1_b$ are homogeneous for $\varphi(x_0,\ldots,x_k,y)$. Using definability of $U$ we get a formula $\psi(x,y)$ such that $\psi(M,b)\in U$ equals $Y^0_b$ or $Y^1_b$.

To prove the claim, follow the proof of $2\Rightarrow 4$ in the previous theorem with $(X^b_a)_{a\in M}$ in place of $(X_a)_{a\in M}$ there. The sets $X^*_0,X^*_1$ there equal $\delta_0(M,b),\delta_1(M,b)$ for suitable formulas $\delta_0(x,y),\delta_1(x,y).$ Using again the definability of $U$ gives a formula
$\delta(x,y)$ as desired.
\end{proof}

\begin{lem}\label{lem:unidef} Unbounded, uniformly d-selective ultrafilters in $\D(M)$
are definable.
\end{lem}

\begin{proof} Assume $U$ is unbounded and uniformly d-selective. Given $\varphi(x,z)$ we have to find a formula $\chi(z)$
 that defines $\{a\in M\mid \varphi(M,a)\in U\}$.  Choose
$\varphi'(x,y,z)
$ 
such that 
for every $a\in M$, $\varphi'(x,y,a)$ defines  the characteristic function $f_a$ of $\varphi(M,a)$.  Applying uniform d-selectivity on $\varphi'$ gives  $\psi(x,y)$ such that $X_a:=\psi(M,a)\in U$ and $f_a\res X_a$ is constant or injective.
Since $X_a$ is unbounded and $f_a$ has values in $\{0,1\}$, 
$f_a\res X_a$  is constant. If its value is 1 resp.~0, then  $\varphi(M,a)\in\uf U$ resp.~$\neg\varphi(M,a)\in \uf U$. Thus we can set $\chi(z):=\exists x(\psi(x,z)\wedge\varphi(x,z)) $.
\end{proof}

\subsection{Existence}\label{sec:exist}

As mentioned in the introduction,
Gaifman showed the existence of types of $\PA$ all of whose extensions to any model $M$  are unbounded and d-selective; same for definable types (see \cite[\S 2]{gaif2} or \cite[Theorems 2.10 and 3.6]{gaif}). His constructions are based on general results on how to get uniform definitions for certain recursively defined sequences of sets \cite[$\S 1$]{gaif2}. Poizat \cite[Theorem 11.9]{poizat} constructs a definable ultrafilter in $\D(M)$ (fixed $M$) based on partial truth definitions. 
We give a self-contained construction of stronger ultrafilters:

\begin{thm}\label{thm:exist} Assume $L$ is countable. Then every unbounded $X\in\D(M)$ is contained in some unbounded,  uniformly d-selective ultrafilter in $\D(M)$.
\end{thm}

\begin{proof} Let  $\varphi(x,y,z)$ be an $L$-formula (with at most the variables shown and without parameters from~$M$). The first and main step is to construct a uniformly definable non-increasing family $(X_a)_{a\in M}$ such that $X_0\subseteq X$ and $X_a$ witnesses the d-selectivity property for the function defined by $\varphi(x,y,a)$ (if it does define a function). We first define $(X^s_a)_{a<\ell}$ using a length $\ell:=\textit{lh}(s)$ sequence $s\in M$ as auxiliary information - suitable choices of $s$ will determine the desired sequence. To explain the idea,
think of $(X^s_a)_{a<\ell}$ as a binary $\ell\times M$-matrix whose lines are the characteristic sequences of $X^s_0$, $X^s_1$,\ldots; so $x\in X_a$ means the $\langle a,x\rangle$th entry is 1 (for readability we abuse notation and use $x,y$ also for elements of $M$). This entry is going to be determined by the other entries in the upper left $[a{+}1]\times [x{+}1]$-submatrix. We aim to compute it by a function $F:M^2\to \{0,1\}$ obtained by course-of-value recursion. 

More precisely, the $\langle a,x\rangle$th entry is $F(\langle a,x\rangle,s)$ and is  computed from the values  \linebreak $F(\langle a{-}1,0\rangle,s), \ldots ,F(\langle a{-}1,x\rangle,s)$  as follows:
\begin{enumerate}\itemsep=0pt
\item[($F$a)] if $(s)_a=0$, then $F(\langle a,x\rangle,s)=F(\langle a{-}1,x\rangle,s)$;
\item[($F$b)] if $(s)_a=1$, then $F(\langle a,x\rangle,s)=1$ if and only if $F(\langle a{-}1,x\rangle,s)=1$ and
there exists $y\in M$ such that $M\models\varphi(x,y,a)$ and for all $z<x$ we have $M\models\neg\varphi(z,y,a)$ or $ F(\langle a{-}1, z\rangle, s)=0$.
\item[($F$c)] if $(s)_{a}=y{+}2$ for some $b\in M$, then $F(\langle a,x\rangle,s)=1$ if and only if $F(\langle a{-}1,x\rangle,s)=1$ and $M\models\varphi(x,y,a)$;
\end{enumerate}
 in case $a=0$ we agree that $F(\langle -1,x\rangle,s)=1$ if and only if $x\in X$. Clearly, $F$ is definable and we let $X^s_a:=\chi(M,a,s)$ for  the formula 
 $$
 \chi(x,y,z)\ :=\ F(\langle y,x\rangle,z)=1.$$ 
 
 \noindent{\em Claim 0:} $(X^s_a)_{a<\ell}$ is non-inreasing and $X^s_0\subseteq X$; further, $X^s_a=X^{t}_a$ for every sequence $t\in M$ extending $s$.
\medskip
 
 \noindent{\em Proof of Claim 0:}  
Let $a,s\in M$ with $a<\ell$. By ($F$a)-($F$c) we have that
 $F(\langle a,x\rangle,s)=1$ implies $F(\langle a{-}1,x\rangle,s)=1$, i.e., $X^s_a\subseteq X^s_{a{-}1}$; by induction in $M$, $(X^s_a)_{a<\textit{lh}(s)}$ is non-increasing. 
 
$X^s_0\subseteq X$ follows from our notational convention for $F(\langle -1,x\rangle,s)$.

Further: let $(s)_{\le a}\in M$ code the sequence $((s)_0,\ldots,(s)_a)$ in $M$. It suffices to show $F(\langle a,x\rangle,s)=F(\langle a,x\rangle,(s)_{\le a})$ for all $x,a<\textit{lh}(s)$. This is clear by induction on $a$ in $M$ because $F(\langle a,x\rangle,s)$ is computed from $(s)_a$ and earlier values of $F$.
 \hfill$\dashv$\medskip 

We now fix the auxiliary information $s\in M$. For every $\ell\in M$ there exists a unique sequence $s\in M$ of length $\ell$ such that for all $a<\ell$,
\begin{enumerate}\itemsep=0pt
\item[($s$a)] $(s)_a=0$ if $\varphi(x,y,a)$ does not define a function in $M$;
\item[($s$b)] $(s)_a=1$ if $\varphi(x,y,a)$ defines a function $f_a:M\to M$ and 
$X^s_{a{-}1}\cap f^{-1}_a(y)$ is bounded for all $y\in M$ (as before, we understand $X^s_{-1}:=X$);
\item[($s$c)] $(s)_a=y{+}2$ if $\varphi(x,y,a)$ defines a function $f_a:M\to M$ and $y\in M$ is minimal such that
$X^s_{a{-}1}\cap f^{-1}_a(y)$ is unbounded.
\end{enumerate}
Let $g\in\F(M)$ map $\ell$ to such an $s$ and define  $X_a:=X^{g(a{+}1)}_a=\chi(M,a,g(a{+}1))$.

 \medskip
 
 \noindent{\em Claim 1:} $(X_a)_{a\in M}$ is non-increasing and unbounded; further,
 $X_0\subseteq X$. 
\medskip
 
 \noindent{\em Proof of Claim 1:}  Using Claim 0, we have $X_0=X^{g(1)}_0\subseteq X$ and, since $g(a{+1})$ extends $g(a)$,
$$
X_a=X^{g(a{+}1)}_a\subseteq X^{g(a{+}1)}_{a{-}1}=X^{g(a)}_{a{-}1}=X_{a{-}1}.
$$ 
 
  We are left to show unboundedness.
 By induction in $M$, it suffices to show that $X_a$ is unbounded assuming $X_{a{-}1}$ is. Let 
$s:=g(a{+}1)$. Then $X_a=X^s_a$ by definition and $X_{a{-}1}=X^s_{a{-}1}$ by Claim 0 ($s$ extends $g(a)$).

If $(s)_a=0$, then $X^s_a=X^s_{a-1}$ is unbounded (by ($F$a)). If $(s)_a=y{+}2$ for some $y\in M$, 
then $X^s_a=X^s_{a-1}\cap f^{-1}(y)$ is unbounded (by ($Fc$) and ($s$c)). 

If $(s)_a = 1$, then $\varphi(x, y, a)$ defines a function $f_a$ such that $f_a\res X^s_{a - 1}$ is bounded-to-one (by~($s$b)). By ($F$b), $X^s_a$ is the set of minima of those sets $X^s_{a - 1} \cap f^{-1}(y)$ which are non-empty. This set is unbounded: 
otherwise $X^s_a\subseteq[b]$ for some $b\in M$; then 
$f\res X^s_{a-1}$ takes values $\le c$ with $c:=\max_{x<b}f(x)\in M\equiv \N$; thus $f\res X^s_{a-1}$ is bounded-to-one with bounded image, so its domain $X^s_{a-1}$ is  bounded - contradicting the induction hypothesis.
\hfill$\dashv$\medskip 
 
 \noindent{\em Claim 2:} If $a\in M$ is such that $\varphi(x,y,a)$ defines a function $f_a$, then
 $f_a\res X_a$ is constant or injective.
 
\medskip
 
 \noindent{\em Proof of Claim 2:} for $s:=g(a{+}1)$ we have $(s)_a\neq 0$. If $(s)_a=y{+}2$ for some $y\in M$, then $X_a=X^s_a=X_{a-1}\cap f_a^{-1}(y)$ (as noted above), so $f_a\res X_a$ is constantly $y$.

If $(s)_a=1$, then  $X_a^s$ is the set of minima of those sets $X^s_{a - 1} \cap f^{-1}(y)$ which are non-empty (as noted above). This means $f_a\res X_a$ is injective.
  \hfill$\dashv$\medskip

  By Claim 1 and Lemma~\ref{lem:psint}, $(X_a)_{a\in M}$ has a pseudo-intersection $Z_0\in\D(M)$. Any unbounded ultrafilter that contains $Z_0$ contains all $X_a,a\in M$. By
 Claim 2, then  $\chi(x,z,g(z{+}1))$ witnesses uniform d-selectivity for $\varphi(x,y,z)$.

Repeat the construction for the next $L$-formula $\varphi'(x,y,z)$ and $Z_0$ in place of $X$ and get an unbounded $Z_1\subseteq^* Z_0$. Going through all formulas gives 
$$
X\supseteq^* Z_0\supseteq^* Z_1\supseteq^* \ldots
$$ 
in $\D(M)$, all unbounded (here we use that $L$ is countable).
By Lemma~\ref{lem:uf}, there exists an unbounded ultrafilter $\uf U$ containing $X$ and all $Z_n,n\in\N$.
Then $\uf U$ is uniformly d-selective.
\end{proof}

 \subsection{Combinatorial ultrapowers}

As mentioned in the introduction, it is long known that combinatorial ultrafilter properties relate to properties of the corresponding definable ultrapower (cf.~\cite{blass2,blass1}). The following collects what we are going to use and is essentially known.
 It  is proved by standard arguments from the model theory of $\PA$.

\begin{prop} \label{prop:ultra} Let $U$ be an 
ultrafilter in $\D(M)$. Then $M^*:=\F(M)/U$
 is an elementary extension of $M':=\{\alpha^U_a\mid a\in M\}\cong M$. Further,
\begin{enumerate}\itemsep=0pt
\item If $U$ is unbounded, then $M'\preceq M^*$ is not cofinal.
 \item If $U$ has the pfip, then $M'\preceq M^*$ is an end extension.
 \item If $U$ is definable, then $M'\preceq M^*$ is {\em conservative}: $X\cap M'\in\D(M')$ for all $X\in\D(M^*)$.
 \item If $U$ is d-selective, then $M'\preceq M^*$ is {\em minimal}:  $N=M'$ or $N=M^*$
 for all $M'\preceq N\preceq M^*$. 
\end{enumerate}
\end{prop}

\begin{proof} 
 $M'\preceq M^*$ is well-known; in fact,  for $F:=\F(M)$ the equivalence of Lemma~\ref{lem:PV} holds  for all  formulas $\varphi(\bar x)$ (not necessarily quantifier-free). 

(1): for the identity $\textit{id}$ on $M$ we have $M^*\models \alpha_a^U<\textit{id}^U$ for all $a\in M$ because $M\setminus [a{+}1]=\{x\in M\mid \alpha_a(x)< \textit{id}(x)\}\in U$.

(2): let $\alpha \in\F(M)$ and assume 
$\alpha^U<\alpha_a^U$ in $M^*$. Then $X:=\{b\in M\mid \alpha(b)< a\}\in U$. Let $f\in \F(M)$ to map $x\in X$ to $\alpha(x)$ and $x\notin X$ to $a{+}1$. Then $f$ takes values  $\le a{+}1$. By Lemma~\ref{lem:psfip}, there is $c\le a{+}1$ with $f^{-1}(c)\in U$. Since $f^{-1}(a{+}1)=M\setminus X\notin U$, we have $c\le a$. Then $f^{-1}(c)=\alpha^{-1}(c)$ and
$\{b\in M\mid\alpha(b)=c \}\in U$, so $\alpha^U=\alpha_c^U\in M'$.

(3): given $\alpha\in\F(M)$ and an $L$-formula $\varphi(x,y)$ 
we have to define $M'\cap\varphi(M^*,\alpha^U)$ in $M'$, or equivalently, $\{a\in M\mid M^*\models\varphi(\alpha^U_a,\alpha^U)\}$ in $M$.
But for all $a\in M$,
$$
M^*\models\varphi(\alpha^U_a,\alpha^U) \ \Longleftrightarrow \ \{b\in M\mid M\models\varphi(a,\alpha(b)) \}\in U
\ \Longleftrightarrow \ M\models \textup{d}_Uy\varphi(a,\alpha(y)).
$$

(4): 
assume $M'\preceq N\preceq M^*$ and $N\neq M'$. Choose  $\alpha\in \F(M)$ with
$\alpha^{\uf U}\in N\setminus M'$. Choose $X\in U$ such that
$\alpha\res X$ is constant or injective. But it is not constant -- otherwise $X\subseteq\{x\in M\mid \alpha(x)=\alpha_a(x)\}\in\uf U $ for some $a\in M$, so $\alpha^{\uf U}=\alpha_a^{\uf U}\in M'$. 
Hence, $\alpha\res X$ is injective. 

We show $\beta^{\uf U}\in N$ for every
 $\beta\in \F(M)$. 
There is
$g\in \F(M)$ such that $(g\circ \alpha)\res X$ is the identity on $X$. Then
$
X\subseteq\big\{ x\in M\mid M\models \beta(g(\alpha(x))=\beta(x) \big\}\in\uf U.
$
The Lemma~\ref{lem:PV} equivalence for the (not necessarily quantifier-free) formula $\beta(g(u))=v$ implies $M^*\models \beta(g(\alpha^{\uf U}))=\beta^{U}$. But $N\equiv M$ is closed under $\beta\circ g$, so
$\beta^{\uf U}\in N$.  
\end{proof}

\section{Local witness models}\label{sec:lwm}

In this section we prove a strengthening of Theorem~\ref{thm:lwm}, 
so we assume $L=\PV$. Let $R\in\NP$ and let $g$ be a generator with bitgraph in $\NP$. The {\em bitgraph of} $g$
is the set
$$
\big\{\langle x,i,b\rangle\in\N\mid i\le |x|,b\in\{0,1\}, \bit(g(x),i)= b\big\}.
$$

We define the formula ``$g(x)\in R$'' from the introduction. We fix, arbitrarily, for every $R\in\NP$ 
a $\PV$-term $t_R(x)$ and a quantifier-free $\PV$-formula $V_R(x,y)$ such that 
$$
\exists y{<}t_R(x)\ V_R(x,y)
$$ 
defines~$R$ in $\N$. For notational simplicity, 
we choose a term $t_g(x)$ and 
a quantifier-free   $V_g(z,y)$ such that 
$\langle x,i,b\rangle$ for $x\in\N,i\le |x|,b<2$ belongs to 
the bitgraph of $g$ if and only if $\N\models \exists v{<}t_g(x)V_g(\langle x,i,b\rangle,v)$.
Then $\{x\in\N\mid g(x)\in R\}$ is defined in $\N$ by
$$
\begin{array}{lcl}
\text{``}g(x)\in R\text{''}& :=& \exists y{<}2^{|x|+1}\ \ \exists w{<}t_R(y)\ \forall i{\le}|x|\ \exists v{<}t_g(x)\ \big(V_g(\langle x,i,\bit(y,i)\rangle,v)\wedge V_R(y,w) \big).
\end{array}
$$

\begin{rem}\label{rem:coll} Assume $N$ is a local witness model for $g$ and $R$ such that
$$
N\models\forall\bar x\forall x \big(\forall i{\le}|x|\ \exists y{<}t(x)\ \varphi(\bar x,x,i,y)
\to\exists z\ \forall i{\le }|x|\ \varphi(\bar x,x,i,(z)_i)\big)
$$
for all quantifier-free $\PV$-formulas $\varphi(\bar x,x,i,y)$ and $\PV$-terms $t(x)$. Then
$$
R\cap\im(g)\neq \emptyset.
$$
\end{rem}

\begin{proof} $\forall\PV$ plus the displayed instances of collection proves $(\text{``}g(x)\in R\text{''}\leftrightarrow\varphi(x))$ for some existential $\varphi(x)$. 
By $N\models\forall\PV$ there exists an extension $N\subseteq N'\equiv \N$. 
Then $N\models\exists x \varphi(x)$ implies $N'\models\exists x\varphi(x)$. As $N'\equiv\N$, also $\N\models\exists x\varphi(x)$ and $\N\models\exists x\text{``}g(x)\in R\text{''}$.
\end{proof}

\subsection{Proof of Theorem~\ref{thm:lwm}}

The following implies Theorem~\ref{thm:lwm}. 

\begin{thm}\label{thm:stronglwm}
 Let $R,g$ be as in Theorem~\ref{thm:lwm}. Then there exist a proper, conservative, minimal elementary end extension $M^*$ of $M$ and a local witness model $N$ for 
$g$ and $R$ such that $$M\prec_1 N\subseteq M^*.$$
\end{thm}

\begin{proof}
We are looking for an unbounded, uniformly d-selective 
ultrafilter~$\uf U$ in~$\D(M)$ such that for a suitable family $F\subseteq\F(M)$
\begin{equation*}\label{eq:goal}
N:=F/\uf U\models \exists x\ \text{``}g(x)\in R\text{''}.
\end{equation*}
Then  we are done: by Lemmas~\ref{lem:unidef} and \ref{lem:dfblepfip},   $U$ has the properties of Proposition~\ref{prop:ultra}. Hence, $M\cong M'\prec M^*:=\F(M)/U$ is a conservative, minimal end extension (this is somewhat redundantly stated since conservative extensions are always end extensions).

\medskip

To find $U,F$ we use the complexity theoretic assumptions on $R,g$.
 If  
 $R\cap\im(g)\neq\emptyset$, then this is also true in $M^*\equiv M\equiv \N$ and we can take $F:=\F(M)$.
 So assume  $R\cap\im(g)=\emptyset$.

First work in the standard model $\N$.
For every $n\in\N$ let 
$b_n$ be the minimal element of length $n{+}1$  in $R$ if such an element exists;
if no such  element exists, set $b_n:=0$.  For every $n\in\N$ let $w_n$ be minimal with $V_R(b_n,w_n)$ if $b_n\in R$, and set $w_n:=0$ otherwise. 
Clearly, $n\mapsto b_n$ and $n\mapsto w_n$ are definable (in $\N$).

For $x\in\N$ let $v_x$ be minimal such that for all $i\le |x|$ we have either
$V_g(\langle x,i,\bit(b_{|x|},i)\rangle,(v_x)_i)$ or $V_g(\langle x,i,1{-}\bit(b_{|x|},i)\rangle,(v_x)_i)$. That is, $(v_x)_i$ witnesses either
$\bit(g(x),i)=\bit(b_{|x|},i)$ or $\bit(g(x),i)\neq\bit(b_{|x|},i)$ is true (in $\N$).
 In the former case we call $(v_x)_i$ {\em good}. Clearly, $x\mapsto v_x$ is definable.

Consider terms in the language $\PV$ plus new constants~$\dot b,\dot w$. There is a definable (in $\N$) function
$(s,n)\mapsto C_n^s$ that maps a code $s$ of such a term $ t(x_1,\ldots, x_r)$ (with variables $x_i$) to a (code of a) circuit $C_n^s$
such that $C^s_n(x_1,\ldots, x_r)$  is the value of $t$ on  numbers $x_1,\ldots, x_r$ of length~$n$. This {\em value} is straightforwardly defined understanding the constants $\dot b,\dot w$ to denote $b_n$ and $w_n$. 
Since $|b_n|,|w_n|\le n^c$ and functions in $\PV$ are computable in polynomial time there exists
 $d_s\in\N$ such that for all $n>1$ 
 the size of $C^s_n$ is $\le n^{d_s}$.


Call $s\in\N$ a {\em special student} if $s$ codes a sequence of codes of such terms, say $s_0,\ldots, s_{c_s}$ for $c_s:=\textit{lh}(s){-}1$ such that $s_i$ codes a term $t_i(x,z_0,\ldots, z_{i{-}1})$ (with at most the variables shown).

 The {\em computation} of $s$ on a number $x$ of length $n$ is the sequence $y$ of length $c_s{+}1$ defined as follows: 
\begin{enumerate}\itemsep=0pt
\item[--] let $(y)_0$ be the value of $t_0(x)$ on (the number) $x$ and set $i_0:=\min\{(y)_0,|x|\}$; 
\item[--] let $(y)_1$ be the value of $t_1(x,(v_x)_{i_0})$ and set $i_1:=\min\{(y)_1,|x|\}$; 
\item[--]\ldots
\item[--] let $(y)_{c_s}$ be the value of $t_{c_s}(x,(v_x)_{i_0},(v_x)_{i_1},\ldots,(v_x)_{i_{c_s-1}})$; this is the {\em output of $y$}.
\end{enumerate}

The computation~$y$ is {\em good} if all $(v_x)_{i_j}, j< c_s,$ are good. 
Clearly, the set of codes of special students and the set of $(s,x,y,z)$ such that $y$ is a (good) computation of $s$ on $x$ with output~$z$ are definable.

It follows from the hardness assumption (H) that for all special students $s$ and all sufficiently large $n$  (depending on $s$) there exists $x$ of length $n$ such that the computation of $s$ on $x$ is good. Given $s$, let $H_s$ be the set of such $x$.

Since $R$ is unbounded, so is $X_0:=\{x\in \N\mid b_{|x|}\neq 0\}$.
For $a>0$ set 
$$\textstyle
X_a:=X_0\cap \bigcap_{s< a}H_s,
$$ 
where $s$ ranges over codes of special students (we agree  that  $\bigcap\emptyset:=\N$). Then $(X_a)_{a\in\N}$ is uniformly definable and non-increasing. 

We claim that every $X_a$ is unbounded. For this we define the composition $s$ of finitely many special students $s_0,\ldots, s_{k}$  as follows: let $s_i, i\le k,$ code the sequence of terms $t_{ij}(x,z_0,\ldots, z_{j{-}1})$ for $j\le c_{i}:=\textit{lh}(s_i){-}1$; for $0<i\le k$ copy variables $z_j$ to $z_{d_i+j}$ where $d_i:=\sum_{\ell<i}c_{\ell}$; let $s$ code the sequence $(t_{00},\ldots, t_{0(c_0-1)},t_{10},\ldots, t_{1(c_1-1)},\ldots, t_{k0},\ldots, t_{k c_{k}})$. Then the computation of $s$ on $x$ is good if and only if the computations of all $s_i,i\le k,$ on $x$ are good. 

Given $y$ we find $x>y$ in $X_a$ as follows: let $s$ be the composition of all special students $<a$ and choose $n>|y|$ with $b_n\neq	 0$ and sufficiently large such that $H_s$ contains some $x$ of length~$n$; then $x>y$ is in $X_a$.




\medskip

Now work in $M$. Special students and their computations are explained in $M\succeq\N$ via their definitions. We let $F$ consist of all functions $\alpha:M\to M$ such that for some special student $s\in M$ we have for all $x$ that $\alpha(x)$ is the output of the computation of $s$ on $x$.

It should be clear that $F$ is $L(M)=(\PV\cup M)$-closed. Indeed, $\alpha_a\in F$ for every $a\in M$ because there is a  special student in the sense of $M$ that constantly outputs~$a$ (this holds in $M\equiv\N$ because it holds in $\N$).

Applying the uniform definition of $(X_a)_{a\in \N}$ in $M$ gives  $(X_a)_{a\in M}$ in $\D(M)$ and, by $\N\equiv M$, it is non-increasing and unbounded. By Lemma~\ref{lem:psint}, it has a pseudo-intersection $X\in\D(M)$. By Theorem~\ref{thm:exist} there is an unbounded, uniformly d-selective 
ultrafilter $\uf U$ containing~$X$. 

Thus, $X_a\in\uf U$ for all $a\in M$.

\medskip

 We are left to show $F/\uf U\models\exists x\ \text{``}g(x)\in R\text{''}$. Note the identity $\textit{id}$ on $M$ is in $F$  via the one-term special student $(x)$. Via their definitions we
 have in $M$ the functions  $x\mapsto b_{|x|},x\mapsto w_{|x|}$ (for $x\in M$); we denote them by $\beta,\omega$. Then $\beta,\omega\in F$ via the one-term special students $(\dot b),(\dot w)$. We claim that $\textit{id}^{\uf U},\beta^{\uf U},\omega^{\uf U}$ witness $x,y,w$ in $\exists x\ \text{``}g(x)\in R\text{''}$.
 
 Since $M\models V_R(b_{|x|},w_{|x|})$ for all $x\in X_0\in\uf U$, we have
 $F/\uf U\models V_R(\beta^{\uf U},\omega^{U})$ (Lemma~\ref{lem:PV}). 
 Since $|\beta(x)|=|x|{+}1$ and 
 $\omega(x)< t_R(\beta(x))$ for all $x\in X_0\in U$, we have
 $\beta^U<2^{|\textit{id}^U|+1}$ and  $\omega^U<t_R(\beta^U)$
  in~$F/\uf U$.
 Let $\iota\in F$ such that $\iota^{\uf U}\le |\textit{id}^{\uf U}|$ in $F/\uf U$, i.e., 
 $$
 Y:=\big\{x\in M\mid M\models \iota(x)\le |x|\big\}\in \uf U.
 $$ 
 
We have to show
$$
F/\uf U\models \exists v{<}t_g(\textit{id}^U)\ V_g(\langle \textit{id}^{\uf U},\iota^{\uf U},\bit(\beta^{\uf U},\iota^{\uf U})\rangle,v).
$$ 
 Assume $\iota$ is given by the special student $s\in M$.   
Let $\gamma\in F$ be given  by the student $\tilde s$ with $\textit{lh}(\tilde s)=\textit{lh}(s){+}1$ obtained from $s$ by appending the (code of the) term $z_{\textit{lh}(s) - 1}$. 
For $x\in Y$ the output $i$ of $s$ on $x$  satisfies $i\le |x|$; note $i=\iota(x)$. The output of $\tilde s$ on $x$  is $\gamma(x)=(v_x)_i$. Then 
$\gamma(x)<t_g(x)$ for all $x\in Y\in\uf U$ and hence
$$
F/\uf U\models \gamma^{\uf U}<t_g(\textit{id}^U) .
$$


For every
 $x\in X_{\tilde s+1}\in\uf U$ we have that $\gamma(x)=(v_x)_i$ is good, i.e.,
$$
Y \cap X_{\tilde s+1}\subseteq\big\{x\in M\mid M\models V_g(\langle x,\iota(x),\bit(b_{|x|},\iota(x))\rangle,\gamma(x))\big\}.
$$
It follows that the latter set is in $\uf U$ and thus $F/\uf U\models V_g(\langle \textit{id}^{\uf U},\iota^{\uf U},\bit(\beta^{\uf U},\iota^{\uf U})\rangle,\gamma^{\uf U})$.
\end{proof}

\subsection{First corollary}

It is natural to ask for  a  local witness model for $g$ and  all infinite $R\in\NP$ simultaneously.  We observe such a model exists under
a slight strengthening of (H). It states that it is hard, given finitely many tasks $(\textup{T}_b)$, to solve at least one of them. This is  
similar  to $\bigvee$-hardness defined on the propositional side in \cite{krabull} (or \cite[Section~4.2]{krapcg}) -- there the use is different, namely in order to amplify the stretch of a generator.

Let $k\in\N$. For $b_0, \ldots, b_k \notin\im(g)$ of lengths $|b_j|=n_j{+}1$ (all $j\le k$) 
the task $(\textup{T}_{b_0,\ldots, b_k})$ is to find on input $(x_0,\ldots, x_k)$ with $|x_j|=n_j$ (all $j\le k$)
some $(i_0,\ldots, i_k)$ with $i_j\le n_j$ (all $j\le k$) such that $\bit(b_j,i_j)\neq\bit(g(x_j),i_j)$ for {\em some} $j\le k$. In the student-teacher computation the student computes as candidate solutions $(k{+}1)$-tuples $(i_0,\ldots, i_k)$ as above and the teacher answers with a counterexample
$(v_0,\ldots ,v_k)$, namely every $v_j$ witnesses $\bit(b_j,i_j)=\bit(g(x_j),i_j)$. If no such counterexample exists, the student succeeds. If this happens  within $c$ rounds on every $(x_0,\ldots, x_k)$ as above and all allowed answers of the teacher,  then the student {\em solves $(\textup{T}_{b_0,\ldots, b_k})$ in $c$ rounds}. As before, we understand such students formally as $c$-tuples of circuits.

\begin{quote} 
($\vee$H) \em For all $k,c, d\in\N$, all sufficiently large $n_0, \ldots, n_k$ and all $b_0, \ldots, b_k \notin \im(g)$ with $|b_j| = n_j{+}1$ (all $j\le k$) the following holds. Every student that solves $(\textup{T}_{b_0, \ldots, b_k})$ in $c$ rounds has  size  $\ge (n_0 + \ldots + n_k)^d$.
\end{quote}

\begin{cor}\label{cor:lwm} Assume $g$ is a generator with bitgraph in $\NP$ that satisfies the hardness assumption $(\vee\textup{H})$. Then there exists a model $N \succ_1 M$ that is a local witness model for $g$ and all infinite $R\in  \NP$.
\end{cor}

\begin{proof} We want to
show that the universal diagram of $M$ is consistent with the set of sentences
$\exists x\ \text{``}g(x)\in R\text{''}$ for all infinite $R\in\NP$. By compactness, it suffices to show this for finitely many infinite  $R_0,\ldots, R_k\in\NP$. By Lemma~\ref{lem:PV} is suffices to find 
a suitable $F\subseteq\F(M)$ and an unbounded ultrafilter $U$ in $\D(M)$ such that
$F/U\models \exists x\  \text{``}g(x)\in R_j\text{''}$ for all $j\le k$. 

We can  assume that every $R_j$ is disjoint from $\im(g)$: if, say,
$R_0\cap\im(g)\neq\emptyset$, then
 $\N\models \text{``}g(\ell)\in R_j\text{''}$ for some $\ell\in\N$; but then this  bounded sentence holds in $F/U$ being an end extension of $\N$ (up to $\cong$). 

We proceed as in Theorem~\ref{thm:stronglwm}. Define the functions 
$$
\langle	x_0,\ldots, x_k\rangle\mapsto \langle b^0_{|x_0|},\ldots b^k_{|x_k|}\rangle,\langle w^0_{|x_0|},\ldots,w^k_{|x_k|}\rangle, \langle v^0_{x_0},\ldots v^k_{x_k}\rangle$$ 
such that $b^j_{|x_j|},w^j_{|x_j|},v^j_{x_j}$ are chosen as before for $R_j$.
Given a $(\PV\cup\{\dot b,\dot w\})$-term $s$ and $n_0,\ldots, n_k$, the value of $s$ on inputs $\langle x_0,\ldots,x_k\rangle$ with $|x_j|=n_j$ (all $j\le k$) is computed by a small circuit; the values
of $\dot b$ and $\dot w$ are now $\langle b^0_{n_0},\cdots,b^k_{n_k}\rangle$ and $\langle w^0_{n_0},\ldots, w^k_{n_k}\rangle$. Special students are defined as before and their computations are defined on $\langle x_0,\ldots,x_k\rangle$ as above straightforwardly: the student's candidate solutions on $x=\langle x_0,\ldots, x_k\rangle$ become $(k{+}1)$-tuples, e.g., $i_0$ becomes 
$\langle i_0,\ldots, i_k\rangle$ with $i_j= \min\{ \pi_j((y)_0),|x_0|\}$; here, $\pi_j\in \PV$ denotes the function
$\langle x_0,\ldots, x_k\rangle\mapsto x_j$. A computation is {\em good} if all $i_j,j\le k,$ from all rounds are good w. r. t. $(\textup{T}_{b^j_{n_j}})$ , i.e., $(v^j_{x_j})_{i_j}$ witnesses that
$\bit(b^j_{|x_j|},i_j)=\bit(g(x_j),i_j)$.
Set 
$$X_0\ :=\ \big\{\langle x_1, \ldots, x_k\rangle \in \N \mid b_{|x_j|}^j \neq 0\text{ for all }j \le k\big\}.
$$
For a special student $s$ let $H_{s}$ be the set of $\langle x_1, \ldots, x_k\rangle \in \N $ such that the computation of $s$ on $\langle x_1, \ldots, x_k\rangle$ is good. For $a>0$ set
$$\textstyle
  X_a := X_0 \cap \bigcap_{s < a} H_s,
$$
where  $s$ ranges over special students. 
As before, but using ($\vee$H), one can show that  $(X_a)_{a\in \N}$ is uniformly definable, non-increasing and  unbounded. The definition gives such $(X_a)_{a\in M}$ in $\D(M)$ and Lemma~\ref{lem:uf} gives an unbounded ultrafilter
$U$ in $\D(M)$ that contains all $X_a,a\in M$. 

Given $j\le k$ we have to show $F / U \models \exists x \text{``}g(x) \in R_j\text{''}$. But
$\pi_j^U, \beta_j^U, \omega_j^U$ witness $x, y, w$ where $\beta_j,\omega_j\in F$ are the functions 
$\langle x_0,\ldots, x_k\rangle\mapsto 
 b^j_{|x_j|}, 
 w^j_{|x_j|}$: argue as  in Theorem~\ref{thm:stronglwm}.
\end{proof}

\begin{rem}\label{rem:impl}\
Note that for every $R\in\NP$
$$
N\models\big(\textup{``} R \textit{ is infinite} \textup{''}\ \to\ \exists x\  \text{``}g(x)\in R\text{''}\big).
$$ 
Indeed, for finite $R\in\NP$ we have $R\subseteq[m]$ for some $m\in\N$; this universal statement belongs to $\forall\PV$ and 
 implies $\neg \textup{``} R \textit{ is infinite} \textup{''}$.
\item \cite{krabull} constructs a model of these implications for  the polynomial time computable truth table generator~$g=\textbf{tt}_{2^{k/2},k}$. However, instead of ($\vee$H), this construction is based on
the hypothesis that $\P\subseteq\mathsf{SIZE}(n^d)$ for some $d\in\N$. This hypothesis
goes back to Kolmogorov (cf.~\cite[p.564]{jukna}) and is today typically  dubbed implausible.
  %
 %
  %
%
\end{rem}

\subsection{Second corollary}

We define the uniform version (uH) of  (H). 
A {\em polynomial time student} is a polynomial time  algorithm $\mathbb A$. Given $b\notin\im(g)$ of length 
$n{+}1$ consider the student-teacher computation where in  round $k$ the student candidate solution $i$ is computed from $x$ and previous counterexamples $v_0,\ldots, v_{k-1}$ by $\mathbb A$.
As before,
 $\mathbb A$ {\em solves} (T$_b$) {\em in $c$ rounds} if for all $x$ of length $n$ there is a round $ \le c$
 such that the candidate solution $i$ computed in this round satisfies $\bit(g(x),i)\neq \bit(b,i)$.
\begin{quote} 
(uH) \em For all $c\in\N$, all polynomial time students $\mathbb A$, all sufficiently large $n$, and all $b\notin \im(g)$ of length $n{+}1$: $\mathbb A$ does not solve \textup{(T$_b$)} in $c$ rounds.
\end{quote}

\begin{df}\label{df:inf} $R\in\NP$ is {\em feasibly infinite} if $\forall\PV$ proves
$$
\textup{``} R \textit{ is infinite} \textup{''} \ :=\ \forall z\exists x\exists y{<}t_R(x) (z<x\wedge V_R(x,y)).
$$
\end{df}

The following implies the variant of Theorem~\ref{thm:lwm} with (uH) replacing (H).

\begin{cor}\label{cor:lwm2} Assume $R$ is feasibly infinite and $g$ is a generator with bitgraph in $\NP$ that satisfies the hardness hypothesis \textup{(uH)}. Then there exist a proper, conservative, minimal elementary end extension $M^*$ of $M$ and a local witness model $N$ for 
$g$ and $R$ such that $$M\prec_1 N\subseteq M^*.$$
\end{cor}
\begin{proof} We modify the proof of Theorem~\ref{thm:stronglwm} and look for a suitable 
$F\subseteq\F(M)$ and an unbounded, uniformly d-selective ultrafilter $U$. As before, we can 
assume $R\cap\im(g)=\emptyset$.
 The assumption that $R$ is feasibly infinite and Herbrand's theorem give
$f_0,f_1\in\PV$ such that 
$$
\forall\PV\vdash x<f_0(x)\ \wedge\ f_1(x)<t_R(f_0(x))\ \wedge\ V_R(f_0(x),f_1(x)).
$$
Write $1^{(n)}:=2^{n}{-}1$ and note $|f_0(1^{(n)})|>n$.
Consider the following algorithm on input $x\in\N$:

\begin{enumerate}\itemsep=-2pt
\item  $k\gets 0$
\item while $k<|x|$
\item \quad  if $|f_0(1^{(k)})|=|x|$, halt with output $1^{(k)}$
\item \quad $k\gets k{+}1$
\item halt with output 0
\end{enumerate}

It is clear that this algorithm runs in polynomial time, 
so it  computes a function $\ell(x)\in\PV$. For $n\in\N$
 set $b_n:=f_0(\ell(1^{(n + 1)}))$ if $\ell(1^{(n + 1)})\neq 0$; else set $b_n:=0$. 
For $n\in\N$ set $w_n:=f_1(\ell(1^{(n + 1)}))$ if $b_n\neq 0$; else set $w_n:=0$.
Note that
$(b_n,w_n)\in V_R$ and $w_n<t_R(b_n)$ 
and $|b_n|=n + 1$ if $b_n\neq 0$. 
It is clear that $x\mapsto b_{|x|}$ and $x\mapsto w_{|x|}$ are computable in polynomial time.

 For  $x\in \N$ let $v_x$
 be minimal such that for all $i \le |x|$  either
$V_g(\langle x,i,\bit(b_{|x|},i)\rangle,(v_x)_i)$ or $V_g(\langle x,i,1{-}\bit(b_{|x|},i)\rangle,(v_x)_i)$ is true (in $\N$). In the former case, call $(v_x)_i$ {\em good}. It is clear that $x\mapsto v_x$ is definable (in $\N$).

Let every $s\in\N$ code a ternary function symbol $\hat s\in\PV$ such that $(s,c,x,z)\mapsto \hat s(c,x,z)$ is definable and for each fixed $s\in\N$, $(c,x,z)\mapsto \hat s(c,x,z)$ is polynomial time.
A {\em student} is a pair~$\langle s,c\rangle$ with $c>0$. The
 {\em computation} of $\langle s,c\rangle$ on $x$ is
the length $c{+1}$ sequence $y$ such that
 \begin{enumerate}\itemsep=0pt
\item[--] $(y)_0:=\hat s(0,x,0)$  and  $i_0:=\min\{(y)_0,|x|\}$; 
\item[--] $(y)_1:=\hat s(1,x,(v_x)_{i_0})$ and  $i_1:=\min\{(y)_1,|x|\}$; 
\item[--]\ldots
\item[--] $(y)_{c}:=\hat s(c,x,\langle (v_x)_{i_0},(v_x)_{i_1},\ldots,(v_x)_{i_{c-1}}\rangle)$; this is the {\em output of $y$}.
\end{enumerate}
The computation is {\em good} if all $(v_x)_{i_j},j<c,$ are good. The definition ensures that there is a definable map $(s,c)\mapsto s_c$  such that the computation of the student $\langle s_c,c{+}1\rangle$ on any~$x$ prolongues the one of $\langle s,c\rangle$ on $x$ by one round with output $(v_x)_{i_c}$ for
$i_c:=\min\{(y)_c,|x|\}$.
 Namely, take $s_c\in\N$ such that $\hat s_c(d,x,z):=\hat s(d,x,z)$ if $d\le c$, and else
$\hat s_c(d,x,z):=(z)_{c}$.

It follows from the hardness assumption (uH) that for all  students $\langle s,c\rangle$ and all sufficiently large $n$  (depending on $\langle s,c\rangle$) with $b_n\neq 0$ there exists $x$ of length $n$ such that the computation of $s$ on $x$ is good. 
Given $\langle s,c\rangle$, let $H_{\langle s,c\rangle}$ be the set of such $x$.

Set $X_0:=\{x\in\N\mid b_{|x|}\neq 0\}$ and note $X_0$ is unbounded.
For $a>0$ set 
$$\textstyle
X_a:=X_0\cap \bigcap_{\langle s,c\rangle< a}H_{\langle s,c\rangle},
$$ 
where $\langle s,c\rangle$ ranges over  students (we agree  that  $\bigcap\emptyset:=\N$). Then $(X_a)_{a\in\N}$ is uniformly definable (in $\N$), non-increasing and unbounded. For unboundedness argue as in Theorem~\ref{thm:stronglwm} and define
a {\em composition} $(s,c_0+\cdots +c_k)$ of finitely many students $(s_0,c_0),\ldots,(s_k,c_k)$ such that
  the computation of $s$ on any $x$ is good if and only if the computations of all $s_i,i\le k,$ on $x$ are good. We omit the details.

Let $F\subseteq\F(M)$ consist of the functions $\alpha: M\to M$ that map $x$ to the output of a computation of a student -- students and their computations are explained in $M$ via their definitions in $\N$. As before, $F$ is $(\PV\cup M)$-closed. Applying its definition in $\N$, we have $(X_a)_{a\in M}$ in $\D(M)$ non-increasing and unbounded. As before, there exists an unbounded, uniformly d-selective ultrafilter $U$ in $\D(M)$ that contains all $X_a,a\in M$.

 We are left to show $F/\uf U\models\exists x\ \text{``}g(x)\in R\text{''}$. 
Let $\textit{id},\beta,\omega$ denote  $x\mapsto x,x\mapsto b_{|x|},x\mapsto w_{|x|}$ (for~$x\in M$). Then
$\textit{id},\beta,\omega\in F$ are witnessed by students $\langle s,1\rangle$ for suitable $s\in \N\subseteq M$.
 We claim that $\textit{id}^{\uf U},\beta^{\uf U},\omega^{\uf U}$ witness $x,y,w$ in $\exists x\ \text{``}g(x)\in R\text{''}$. 

Clearly, 
$\beta^U<2^{|\textit{id}^U|+1}$ and $\omega^{\uf U}< t_R(\beta^U)$  in $F/\uf U$.
  Let $\iota\in F$ be such that $\iota^{\uf U}\le |\textit{id}^{\uf U}|$ in $F/\uf U$, i.e., 
 $
 Y:=\big\{x\in M\mid M\models \iota(x)\le |x|\big\}\in \uf U.
 $ 
 Assume $\iota$ is given by the student $\langle s,c\rangle\in M$.  We have the function $(s,c)\mapsto s_c$ in $M$ via its definition in $\N$. 
 Let $\gamma\in F$ be given  by the student $\langle s_c,c{+}1\rangle\in M$. 
On $x\in Y\in U$ the output $i:=\iota(x)$ of (the computation of) $\langle s,c\rangle$ on $x$ satisfies $i\le |x|$; the output of $\langle s_c,c{+}1\rangle$ on $x$  is $\gamma(x)=(v_x)_i$. As $|(v_x)_i|\le t_g(x)$ for all $x\in Y\in\uf U$, we have
 $\gamma^{\uf U}<t_g(\textit{id}^U)$ in $F/U$.
 
Set $a:=\langle s_c,c{+}1\rangle{+}1\in M$.  
For every $x\in X_a\in\uf U$ we have that $\gamma(x)=(v_x)_i$ is good, i.e.,
$$
Y \cap X_{a}\subseteq\big\{x\in M\mid M\models V_g(\langle x,\iota(x),\bit(b_{|x|},\iota(x))\rangle,\gamma(x))\big\}.
$$
It follows that the latter set is in $\uf U$ and thus $F/\uf U\models V_g(\langle \textit{id}^{\uf U},\iota^{\uf U},\bit(\beta^{\uf U},\iota^{\uf U})\rangle,\gamma^{\uf U})$.
\end{proof}

\section{Conclusion}
As explained in the introduction, the model theory of local witness models is of central importance in proof complexity. The only constructions  known prior to this work  are due to Kraj\'i\v{c}ek and give cofinal extensions of suitable cuts of a nonstandard $M\equiv\N$. Dually, this work constructs local witness models that are end extensions of $M$.

This is done using techniques from set theory and the model theory of $\PA$. Namely, we transfer some of the basic set theory of combinatorial ultrafilters from $\N$ to $M$. Our local witness model is a definable (sub)ultrapower based on a combinatorial ultrafilter and can be analyzed by techniques familiar from the model theory of $\PA$.

We find it encouraging that such techniques from neighboring fields can be meaningfully   imported to proof complexity. 

\paragraph{Questions}

As explained in the introduction, the main open question is how to construct local witness models satisfying 
the instances of sharply bounded collection in Remark~\ref{rem:coll}. 

Let us state explicitly some  questions concerning the model $N$ of Corollary~\ref{cor:lwm}. Note it satisfies  the implications of Remark~\ref{rem:impl} only for {\em standard} $R\in\NP$. We do not know whether one can strengthen this to all $R\in\NP(N)$, i.e., allowing parameters from $N$ in the formulas defining~$R$ (see~\cite[Definition~15.3.1]{krabuch}).
  We also do not know how to construct  $N$ as an end extension of $M$. 
Further, it would be interesting to relax the definition of feasible infinity to theories stronger than $\forall\PV$, e.g., add $\Sigma^b_1$-induction (see~\cite[Section 10.3]{krapcg}).

Concerning 
ultrafilter theory, Section~\ref{sec:uf} does not give a complete picture and, in particular, lacks counterexamples.
We include one for illustration:

\begin{ex} Assume $L$ is countable. 
Then there exists an unbounded,  definable ultrafilter in $\D(M)$ that is not a d-Q-point. \end{ex}

\begin{proof} This is a modification of a construction of Blass~\cite[p.162]{blass2}. Let $U$ be an unbounded  definable ultrafilter (Theorem~\ref{thm:exist} and Lemma~\ref{lem:unidef}). One straightforwardly checks that 
$$
 V\ :=\ \big\{  X\in\D(M)\mid M\models \textup{d}_{U}  x\ \textup{d}_{U} y\ ( \langle x,y\rangle\in X)    \big\}$$
%
%
is an unbounded, definable ultrafilter. We show $V$ is not a d-Q-point. Note $X:=\{\langle a,b\rangle\mid a<b\}\in V$ because $\{b\in M\mid a<b\}\in U$ for all $a\in M$ (as $U$ is unbounded), so $\textup{d}_{U}y\ x<y$ defines $M\in U$, so $M\models\textup{d}_{U}x\ \textup{d}_{U}y\ x<y$. The function $\langle x,y\rangle\mapsto y$ is bounded-to-one 
on $X$. Let $f\in\F(M)$  extend it by the identity on $M\setminus X$. Then $f$ is bounded-to-one. 

Assume it is 
injective on some $X'\in V$; we can assume $X'\subseteq X$.
By definition of $V$, we have 
$\{a\in M\mid M\models \textup{d}_Uy(\langle a,y\rangle\in X')\}\in U$. 
Let $a,a'$ be distinct elements of this set ($U$ is unbounded), so $Y:=\{b\in M\mid\langle a,b\rangle\in X' \}\in U$ and $ Y':=\{b\in M\mid\langle a',b\rangle\in X' \}\in U$. For $b\in Y\cap Y'\in U$ 
we have $\langle a,b\rangle,\langle a',b\rangle\in X'$. By $X'\subseteq X$, we have
 $f(\langle a,b\rangle)=f(\langle a',b\rangle)=b$, so $f$ is not injective on $X'$, a contradiction.
%
\end{proof}

For example, we do not know whether d-Ramsey ultrafilters can be undefinable.

\begin{rem} 
Unbounded, d-Ramsey ultrafilters $U$ in $\D(M)$ are {\em almost definable}: for every $L(M)$-formula $\varphi(x,y)$ there exists $X\in U$ such that  $\{a\in X\mid \varphi(a,M)\in U\}\in \D(M)$.
\end{rem}
\begin{proof} Given $\varphi(x,y)$, choose $X\in U$ according to the d-Ramsey property. It suffices to observe for every $a \in X$: 
$$ \varphi(a,M)\in U\ \Longleftrightarrow\ M\models \exists y\in X(a<y\wedge\varphi(a,y)). $$ Indeed, $\Rightarrow$ holds because  $\varphi(a,M)\cap X\in U$ is unbounded; $\Leftarrow$: since $a\in X$ either $X\setminus[a+1]\subseteq\varphi(a,M)$ or  $X\setminus[a+1]\subseteq\neg \varphi(a,M)$; since $U$ is unbounded, $X\setminus[a{+}1]\in U$; hence, $\varphi(a,M)\notin U $ implies   $X\setminus[a{+}1]\subseteq\neg \varphi(a,M)$, i.e., $M\models\forall y\in X(a<y\to \neg\varphi(x,y))$.\end{proof}

This notion seems relevant: the reader  easily checks that  $4\Rightarrow 1$  in Theorem~\ref{thm:ramsey} holds for unbounded, almost definable $U$ with the pfip.

\section*{Acknowledgements} We thank Diego Mej\'ia for discussions about ultrafilter theory and Jan Kraj\'i\v{c}ek for comments on an early draft of this paper.

\end{document}